\documentclass[12pt,a4paper,reqno]{amsart}
\usepackage{relsize, enumerate}
\usepackage{hyperref}\hypersetup{linktocpage}

\usepackage{amsmath,amsfonts,amssymb,amsthm,amscd}
\usepackage{MnSymbol}
\usepackage{pst-node,pstricks,pst-tree}
\usepackage[subrefformat=parens]{subcaption}
\usepackage{xcolor}
\usepackage{comment}
\usepackage{graphicx}
\usepackage[all]{xy}
\usepackage{tikz}

\usepackage{geometry}
\makeatletter
\newcommand*{\rom}[1]{\expandafter\@slowromancap\romannumeral #1@}
\makeatother

\renewcommand{\a}{\alpha}

\renewcommand{\l}{\lambda}

\renewcommand{\r}{\rho}

\newcommand{\Ac}{{\mathcal A}}

\newcommand{\Lc}{{\mathcal L}}

\newcommand{\Oc}{{\mathcal O}}
\newcommand{\Pc}{{\mathcal P}}

\newcommand{\Hc}{{\mathcal H}}

\newcommand{\C}{{\mathbb C}}

\newcommand{\N}{{\mathbb N}}
\newcommand{\R}{{\mathbb R}}

\newcommand{\Z}{{\mathbb Z}}
\newcommand{\Q}{{\mathbb Q}}

\def\({\left(}
\def\){\right)}

\def\l\{{\left\{}
\def\r\}{\right\}}
\def\wt{\widetilde}
\def\w{\widehat}
\def\wb{\overline}
\def\id{{\rm id}}

\def\tr{{\rm tr}}
\def\ad{{\rm ad}}
\def\Ad{{\rm Ad}}

\def\Ker{{\rm Ker\,}}
\def\Leb{{\rm Leb}}
\def\ev{{\bf ev}}
\def\diam{{\rm diam\,}}
\def\Pb{{\bf P}}
\def\PSL{{\rm PSL}}
\def\GL{{\rm GL}}
\def\SL{{\rm SL}}

\newtheorem{theorem}{Theorem}[section]
\newtheorem{proposition}[theorem]{Proposition}
\newtheorem{lemma}[theorem]{Lemma}

\newtheorem{conjecture}[theorem]{Conjecture}

\theoremstyle{definition}
\newtheorem{definition}[theorem]{Definition}

\theoremstyle{remark}
\newtheorem{remark}[theorem]{Remark}

\begin{document}

\title[Asymptotic statistics for continued fractions]{Asymptotic statistics for finite continued fractions with restricted digits}

\author{Jungwon Lee}
\address{Max Planck Institute for Mathematics, Vivatsgasse 7, 53111 Bonn, Germany}
\email{jungwon@mpim-bonn.mpg.de}

\date{\today}

\begin{abstract}
Zaremba's conjecture concerns a formation of continued fraction expansions for rational numbers with partial quotient bounded by an absolute constant. We present asymptotic estimates for the size of $\epsilon$-thickening of certain fractal sets of bounded-type, which in turn provide a remark on Zaremba's conjecture in an averaging sense. We also discuss a generalisation for complex continued fractions over imaginary quadratic fields.
\end{abstract}

\maketitle
\setcounter{tocdepth}{1}
\tableofcontents

\section{Introduction}

For $x \in [0,1]$, we denote by
\begin{equation} \label{exp:cf}
x=[0; a_1, a_2, \ldots]=\frac{1}{a_1 +\frac{1}{a_2 +\cdots}}  
\end{equation}
its continued fraction expansion. The integers $a_j \geq 1$ are called the digit or partial quotient of $x$. We remark that the expansion terminates uniquely in a finite time $\ell=\ell(x)$ with $a_\ell \geq 2$ if and only if $x$ is rational.

One of the famous conjectures in Diophantine approximation, predicted by Zaremba in 1971, concerns  analytic structures of finite continued fraction expansions with restricted digits and simply stated as follows.

\begin{conjecture}[Zaremba] \label{conj:Z}
Let $N \in \N$. Then there exists $a \in (\Z/N\Z)^\times$ such that 
\[ \frac{a}{N}=[0; a_1, a_2, \ldots, a_\ell] \]
has all $a_j \leq A$ for some absolute $A \geq 2$.
\end{conjecture}

A substantial progress towards Zaremba's conjecture is given by Bourgain--Kontorovich \cite{bour:kon} that the conjecture holds for $A=50$ and for almost every $N \in \N$ in the sense of density. Despite the completely elementary statement, their proof requires surprisingly heavy machinery from the diverse contexts, in particular application of circle method to count thin group orbits.

To better introduce our result and approach, we discuss a reformulation of Zaremba's conjecture in terms of the Hausdorff dimension $\delta_A$ of a fractal set of real numbers whose digits are at most $A$:
\begin{equation} \label{def:bdd}
E_A=\{ x=[0;a_1,a_2,\ldots] \in [0,1]: a_j \leq A \}.
\end{equation}
Setting 
\[ D_A=\left\{ N \in \N: \frac{a}{N} \in E_A \cap \Q \ \mbox{for some $a\in (\Z/N\Z)^\times$}  \right\}, \]
Conjecture \ref{conj:Z} now reads that $D_A =\N$ for some absolute constant $A \geq 2$. Indeed, Zaremba suggested a sufficient value for $A=5$. 

Further, Hensley claimed a deeper connection between the fractal dimension of the restricted digit set and Zaremba's conjecture as follows. In Conjecture~\ref{conj:Z}, one can allow more generality by requiring the partial quotients $a_j$ to lie in any finite subset $\mathbf{A} \subset \N$ and accordingly consider $E_{\mathbf A}$ and its Hausdorff dimension $\delta_{\mathbf A} \in (0,1)$.

\begin{conjecture}[{\cite[Conjecture 3]{hensley}}] \label{conj:H}
We have $\delta_{\mathbf A}>1/2$ if and only if $D_{\mathbf A} \supset \N_{\gg 1}$.
\end{conjecture}

In particular, for $\mathbf{A}=\{1,2\}$, he observed that $\delta_2>1/2$ and Conjecture \ref{conj:H} would imply the special case of Conjecture~\ref{conj:Z} with $A=2$. In this regard, computing $\delta_A$ has been a crucial theme in the literature, in particular rigorous numerical bounds can be found in Jenkinson--Pollicott \cite{jen:pol, jen:pol2}. Meanwhile, \cite{hensley} considered an extra averaging condition on $D_A$ in Conjecture \ref{conj:H} by setting 
\begin{align} \label{def:Zset:av}
\Sigma_{N,A}&:=\left\{ \frac{a}{N} \in E_A : 1 \leq a < N, (a,N)=1 \right\}  \\
\Omega_{N,A}&:=\left\{ \frac{a}{n} \in E_A : 1 \leq a <n \leq N, (a,n)=1 \right\} \label{def:Zset:av2}
\end{align}
and observed $|\Omega_{N,A}| \sim N^{2\delta_A}$, where the notation `$\sim$' represents the asymptotic limit here and throughout, that is if $f,g$ are real valued functions then $f(x)/g(x) \rightarrow 1$ as $x \rightarrow \infty$. This suggests that we have the following strong form of Zaremba's conjecture.

\begin{conjecture} \label{conj:Z:revisit}
Let $N \in \N$. Then
\[ |\Sigma_{N,A}| \sim \frac{N^{2\delta_A}}{\varphi(N)} \]
where $\varphi$ denotes the Euler phi function.
\end{conjecture}

This shows that, the study towards Conjecture \ref{conj:Z} is closely related to reducing the exponent of $N$ in the asymptotic formulae for the size of \eqref{def:Zset:av} and \eqref{def:Zset:av2}. In particular when $N$ is prime, Conjecture \ref{conj:Z:revisit} reads that $|\Sigma_{N,A}| \sim N^{2\delta_A-1}$ as we have $\limsup_{N \to \infty} \frac{\varphi(N)}{N}=1$ and there are some bounds known due to Moshchevitin--Murphy--Shkredov \cite{Shk19}: For any $\epsilon>0$, there is $A=A(\epsilon)>0$ such that 
\[|\Sigma_{N,A}| \ll N^{2\delta_A-1+\epsilon(1-\delta_A)} \]
by applying combinatorial arguments concerning the bounds for sum-products and Cayley graphs. See also \cite{Shk} for further recent developments towards the conjecture.

In this article, we make a short remark in this direction using a dynamical approach based on transfer operator methods.  We introduce an auxiliary probability space $\Sigma_{N,A}(\epsilon)$ with $\epsilon:=\epsilon(N)>0$ such that $\Sigma_{N,A} \subset \Sigma_{N,A}(\epsilon) \subset \Omega_{N,A}$ and yields an averaging result reducing the exponent of asymptotic statistics for classical as well as complex continued fractions with restricted digits.

\subsection{Main results} \label{int:setup}

Our main result for classical continued fractions concerning Conjecture \ref{conj:Z:revisit} is the following.

\begin{theorem} \label{main:overQ}
We have an auxiliary probability space $\Sigma_{N,A}(\epsilon)$ with $\epsilon=\epsilon(N)>0$ such that $\Sigma_{N,A} \subset \Sigma_{N,A}(\epsilon) \subset \Omega_{N,A}$ and 
\[ |\Sigma_{N,A}(\epsilon)| \sim N^{2\delta_A-\eta} \]
for some $0<\eta<1$.
\end{theorem}

Next we consider an analogue of Theorem \ref{main:overQ} over imaginary quadratic fields based on the work of Kim--Lee--Lim \cite{kim:lee:lim} where they studied dynamics of complex continued fractions and limit theorems for the associated costs. 

We recall that the continued fraction expansion in \eqref{exp:cf} for real numbers is determined by the Gauss map defined as 
\[ T: z \longmapsto \begin{cases}
\frac{1}{x}-\lfloor \frac{1}{x} \rfloor& \hbox{ if } 0<x \leq 1 \\
0& \hbox{ if } x=0
\end{cases} \]
and this can be naturally generalised to complex numbers under the base change over imaginary quadratic fields $\Q(\sqrt{-d})$, $d>0$. This was first introduced by Hurwitz \cite{hurwitz} for $d=1$ then later generalised and studied by Ei--Nakada--Natsui \cite{nakada} for all five norm-Euclidean fields, i.e. for $d \in \{1,2,3,7,11\}$ as follows.

Let $\mathcal{O}_d$ be the ring of integers of $\Q(\sqrt{-d})$, which is a lattics in $\C$:
\[ \mathcal{O}_d =
\begin{cases}
\Z[\sqrt{-d}]& \hbox{ if } d=1,2\\
\Z[\frac{1+\sqrt{-d}}{2}]& \hbox{ if } d =3,7,11.
\end{cases}  \]
This defines a compact fundamental domain $I_d:=\overline{\C / \mathcal{O}_d}$, i.e. set of points whose nearest integer in $\mathcal{O}_d$ are the origin, is either a rectangular or hexagonal polygon due to different $\Z$-module structures of $\mathcal{O}_d$. Then the complex Gauss map $T_d$ on $I_d$ is analogously defined by 
\[ T_d: z \longmapsto \begin{cases}
\frac{1}{z}-\alpha& \hbox{ if } z \neq 0 \\
0& \hbox{ if } z=0
\end{cases} \]
where $\alpha \in \mathcal{O}_d$ is a unique integral element such that $1/z-\alpha \in I_d$. Similarly as in \eqref{exp:cf}, this yields a complex continued fraction expansion for $z=[0;\alpha_1, \alpha_2, \ldots, \alpha_\ell] \in I_d$ and terminates uniquely in a finite time if and only if $z \in \Q(\sqrt{-d})$.

It is then also natural to propose the following version of complex Zaremba conjecture; see also relevant contexts, e.g. \cite{ref1, ref2} for $d=1$. For any $d \in \{1,2,3,7,11\}$ and $z \in \Q(\sqrt{-d})$, we have a reduced form $z=\alpha/\beta$ with relatively prime $\alpha, \beta \in \mathcal{O}_d$. We recall a canonical height function $\mathrm{ht}:  \Q(\sqrt{-d}) \rightarrow \Z_{\geq 0}$ defined by 
\[\mathrm{ht} : z \longmapsto \max\{|\alpha|, |\beta| \}  \]
where $|\cdot|$ denotes the absolute value on $\C$.

\begin{conjecture}[Zaremba over imaginary quadratic fields] \label{conj:Z:quad}
Let $N \in \N$. Then there exist $\alpha, \beta \in \mathcal{O}_d$ with $\mathrm{ht}(\alpha/\beta)^2=N$ such that 
\[ \frac{\alpha}{\beta}=[0; \alpha_1, \alpha_2, \ldots, \alpha_\ell] \]
has all $\alpha_j \in \mathcal{A}_d$ for some bounded set $\mathcal{A}_d \subset \mathcal{O}_d$.
\end{conjecture}

The notion of $E_A, \Sigma_{N,A}$ and $\Omega_{N,A}$ in \eqref{def:bdd}, \eqref{def:Zset:av} and \eqref{def:Zset:av2} can be also extended to complex continued fractions in a canonical way:
\begin{align*} 
E_{\mathcal{A}_d}&=\{ z=[0;\alpha_1,\alpha_2,\ldots] \in I_d: \alpha_j \in \mathcal{A}_d \} \\
\Sigma_{N,\Ac_d}&=\left\{z \in E_{\Ac_d} \cap \Q(\sqrt{-d}) : \mathrm{ht}(z)^2=N \right\}  \\
\Omega_{N,\Ac_d}&=\left\{ z \in E_{\Ac_d} \cap \Q(\sqrt{-d}) : \mathrm{ht}(z)^2 \leq N \right\} .
\end{align*}

Note that similarly as in Conjecture \ref{conj:Z:revisit}, Conjecture \ref{conj:Z:quad} can be implied by the asymptotic $|\Sigma_{N,\Ac_d}| \sim N^{2 \delta_{\Ac_d}-2}$ where $\delta_\Ac=\mathrm{dim}_H(E_{\Ac_d}) \in (1,2)$, since now the set of rationals in $\Q(\sqrt{-d})$ whose height squared equals $N$ is of size $N^2$. We then obtain an analogue of Theorem \ref{main:overQ} for complex continued fractions.

\begin{theorem} \label{main:overK}
Fix $d \in \{1,2,3,7,11\}$ and drop it from the notation. We have an auxiliary probability space $\Sigma_{N,\Ac}(\epsilon)$ with $\epsilon=\epsilon(N)>0$ such that $\Sigma_{N,\Ac} \subset \Sigma_{N,\Ac}(\epsilon) \subset \Omega_{N,\Ac}$ and 
\[ |\Sigma_{N,\Ac}(\epsilon)| \sim N^{2\delta_{\Ac}-\theta} \]
for some $0<\theta<2$.
\end{theorem}

\begin{remark}
We do not present the proof of Theorem \ref{main:overK} separately as it is repeating the identical arguments for the proof of Theorem \ref{main:overQ} in Section \ref{sec:taub} and \ref{sec:smooth}, along with analogous properties of complex continued fractions summarised in Section \ref{subsec:cf:K}; See also  \cite[Section 8]{kim:lee:lim} for relevant details. 

Note that in this case, the Hausdorff dimension $\delta_\Ac$ takes values in 0 and 2, hence the result is slightly weaker than the real case in view of the size of exponent. We presume that there is a possibility to improve the constant $\theta$ with $1<\theta<2$ by taking the scale of $\epsilon(N)$ in a more subtle way.
\end{remark}

\subsection{Overview of the approach}

We note that the proofs of \cite{jen:pol, jen:pol2, hensley} rest on spectral analysis of transfer operators associated to the Gauss map, in particular the compactness of the operator that enables them to understand explicitly the dimension $\delta_A$ in connection with the largest eigenvalue in the spectrum.  

We follow to utilise transfer operators by extending the framework of Baladi--Vallée \cite{bv} for asymptotic Gaussian distribution of the length of continued fractions, where they required more delicate analysis, in particular the Dolgopyat-type uniform polynomial decay of operator norms. This was essential to apply a Tauberian theorem (Perron's formula of order 2) to complex function identified using the resolvent of transfer operator, and accordingly have explicit estimates for the moment generating function of the length as a random variable on the set of rationals with bounded denominator. 

Here we adapt the approach to obtain asymptotic estimates for the size of fractal sets $\Omega_{N,A}$ and $\Omega_{N,\Ac_d}$ closely related to Conjecture \ref{conj:Z} and \ref{conj:Z:quad} by introducing weighted transfer operators, constrained to the sets $A$ and $\Ac_d$. Along with Lee--Sun \cite{lee:sun}, where they made some simplifying remarks on \cite{bv} regarding the use of Perron's formula and smoothing process, we finally construct auxiliary probability spaces in Theorem \ref{main:overQ} and \ref{main:overK}.

\begin{remark}
Our main statement yields an additional evidence for Zaremba's conjecture in terms of the size of fractal set of bounded-type, however, it is in an averaging sense. That means, we could not suggest any statistics for the exact set of rationals with fixed denominator but we obtain the asymptotic formula for the set with $\epsilon$-thickening on the denominator.
 
We would like to emphasise that this observation comes from the interactions with dynamical aspects of continued fractions, in particular as an application of a spectral gap and Dolgopyat estimate for the associated transfer operators. 
\end{remark}

This article is organised as follows. In Section \ref{sec:op}, we introduce constrained transfer operators attached to real and complex continued fractions with restricted digits, and collect their spectral properties. In Section \ref{sec:taub}, we first prove that $|\Omega_{N,A}| \sim N^{2 \delta_A}$ in Theorem \ref{main:overQ} applying a Tauberian theorem to convert results on complex functions to asymptotic statistics for Zaremba-type bounded digit sets. In Section \ref{sec:smooth}, we complete the proof of Theorem \ref{main:overQ}, which relies on the probabilistic smoothing process.

\subsubsection*{Funding} Open Access funding enabled by Max Planck Society.

\subsubsection*{Conflict of Interest} The author declares no conflict of interest.

\subsubsection*{Data Availability} No datasets were generated or analysed during the current study.

\section{Constrained transfer operators} \label{sec:op}

We first introduce and derive spectral properties of constrained transfer operators associated to the Gauss map and complex Gauss map over imaginary quadratic fields. 

\subsection{Classical continued fractions} \label{subsec:cf}

For $a \in \N$, the set $I_a:=\{ x \in [0,1]: \lfloor 1/x \rfloor=a \}$ forms a countable partition of pairwise disjoint subsets such that 
\[ [0,1]=\bigcup_{a \in \N} \overline{I_a} \ \mbox{ and } \  T|_{I_a}:I_a \rightarrow T(I_a) \mbox{ is bijective}. \]
Note that $T(I_a)=[0,1]$ for all $a \in \N$, hence we denote by $\Hc=\{h_a=T|_{I_a}^{-1}: a \in \N \}$ the set of inverse branches of $T$ of the form: $h_a(x)=1/(x+a)$.

We introduce the families of transfer operators. Let $s, w \in \C$. For $A \geq 2$, we denote by $\Hc_A=\{h_a \in \Hc, a \leq A \} \subset \Hc$ the set of inverse branches corresponding to the restricted digit condition.

\begin{definition}
Let $\Lc_{s,w,A}: C^1([0,1]) \rightarrow C^1([0,1])$ be the transfer operator, constrained to finite number of inverse branches, defined by 
\[ \Lc_{s,w,A} f(x)= \sum_{h \in \Hc_A} e^{s \log|h'(x)|+w \cdot 1 \circ h(x)} \cdot (f\circ h)(x),  \]
where 1 denotes the constant function and $\circ$ denotes a composition of the functions.
\end{definition}

Then we have the following key spectral properties of transfer operator acting on $C^1$-functions, which follow rather straightforwardly from the arguments in \cite[Proposition 0, Eq.(1.8)]{bv}. Equip $C^1([0,1])$ with the norm 
\begin{equation} \label{def:norm}
\|f\|_{(t)}:=\|f\|_\infty+ |t|^{-1} \|f'\|_\infty , \qquad t \in \R^\times, 
\end{equation}
where $\|f\|_\infty=\sup_{x \in [0,1]} |f(x)|$. We will use the same notation for the associated operator norm.

\begin{proposition} \label{prop:tr:overQ}
Set $s=\sigma+it, w=u+iv$. 
\begin{enumerate}
\item The operator $\Lc_{\sigma,u,A}$ on $C^1([0,1])$ is quasi-compact with a unique simple eigenvalue $\lambda_{\sigma,u,A}$ of maximal modulus associated to positive eigenfunction $h_{\sigma,u,A}>0$.
\item The operator $\Lc_{s,w,A}$ depends analytically for $(\sigma,u)\sim (\delta_A, 0)$, accordingly $\lambda_{s,w,A}$ and $h_{s,w,A}$ are well-defined and analytic.
\item The resolvent operator satisfies $\| (\mathcal{I}-\Lc_{s,w,A})^{-1}\|_{(t)} \leq \max\{1, |t|^\xi\}$ uniformly in $w$, with $0<\xi<1/5$ and sufficiently large $|t|$.
\end{enumerate}
\end{proposition}

\begin{proof}
For the convenience of the reader, we briefly summarise the key ideas. We refer to \cite{baladi, bv, par:pol} for details. We remark that some arguments could be even simplified due to finiteness of the summand in the definition of constrained transfer operator.

For part (1), the quasi-compactness with existence of $\lambda_{\sigma,u}$ follows from the Lasota--Yorke inequality that is implied by uniform expanding and distortion properties: 
\[ |h'(x)| \ll \rho^n \ \mbox{ and } \ |h''(x)| \ll |h'(x)|, \qquad h \in \Hc^n, n \geq 1, x \in [0,1] \]
with uniform constant $0<\rho<1$.

Part (2) is simply a consequence of analytic perturbation theory. Further for $\Lc_{\sigma,0}$, the dominant eigenvalue $\lambda_{\sigma,0}$ is explicitly given by $e^{P(-\sigma \log|T'|)}$ with a pressure function $P$. Bowen showed that the map $\sigma \mapsto P(-\sigma \log|T'|)$ is convex, strictly decreasing and vanishes exactly at $\sigma=\delta_A=\dim_H(E_A)$. We remark that the Hausdorff dimension takes values between 0 and 1 in this case. 

For part (3), this type of uniform polynomial estimate is called a Dolgopyat estimate; See \cite{dolgopyat, bv}. Though it is now regarded as a standard technique, verifying that this estimate holds is challenging as it needs precise understanding of properties (1) and (2), along with subtle analysis using an additional Uniform Non-Integrability property \cite[Section 3]{bv}.  
\end{proof}

\subsection{Complex continued fractions} \label{subsec:cf:K}

In Kim--Lee--Lim \cite{kim:lee:lim}, they showed that properties of real continued fractions stated in Section \ref{subsec:cf} are analogously generalised to complex continued fractions with suitable modifications as follows. 

Let $(I_d, T_d)$ be the complex Gauss dynamical system as defined in Section \ref{int:setup}. We fix $d \in \{1,2,3,7,11\}$ and drop it from the notation unless otherwise stated. 

For $\a \in \Oc$, let $O_\a:=\{ z \in I : [ 1/z]=\a \}$. We remark that the set $O_\a$ can be empty and $TO_\a \subsetneq I$ for finitely many $\a$, in contrast to the real case. Then non-empty $O_\a$'s form a countable partition for $I$ such that $T|_{O_\a}: O_\a \rightarrow TO_\a$ is bijective. Abusing the notation, we denote by $\Hc_\Ac=\{ h_\a=T|_{O_\a}^{-1} : \a \in \Ac \}$ the set of inverse branches of $T$ corresponding to the digits restricted to some bounded subset $\Ac \subset \Oc$.

Due to the fact that $TO_\a \subsetneq I$ for some $\a$, there are admissibility conditions among digits in the complex continued fraction expansion that make spectral analysis complicated. In particular, one cannot simply use $C^1$-functions on $I$ as one has to consider the function spaces which contain characteristic functions supported on $TO_\a$'s. Based on the work 
\cite{nakada}, we have the technical lemma \cite[Proposition 3.8]{kim:lee:lim}: There is a finite Markov partition $\Pc$ compatible with $T$, i.e. 
\begin{itemize}
\item For each non-empty $O_\a$, $TO_\a$ is a disjoint union of elements in $\Pc$, and
\item For each $h_\a$ and $P \in \Pc$, either there is a unique $Q \in \Pc$ such that $h_\a(P) \subset Q \cap O_\a$ or $h_\a(P) \cap I =\emptyset$.
\end{itemize}

This allows us to consider a direct sum space $C^1(\Pc)=\oplus_{P \in \Pc} C^1(\overline{P})$ of functions $f:I \rightarrow \C$ such that for each $P \in \Pc$, $f|_P$ extends to a continuously differentiable function on an open neighbourhood of $\overline{P}$, equipped with a piecewise $C^1$-norm $\|\cdot\|_{(t)}$ analogously defined as in \eqref{def:norm} by taking the supremum over the finite structure $\Pc$.

\begin{definition}
Let $\Lc_{s,w,\Ac}: C^1(\Pc) \rightarrow C^1(\Pc)$ be the constrained transfer operator defined by 
\[ \Lc_{s,w,\Ac} f(z)= \sum_{h \in \Hc_\Ac} e^{s \log|h'(z)|+w \cdot 1 \circ h(z)} f\circ h(z) \cdot \chi_{\mathrm{Dom}(h)}(z) \]
where 1 denotes the constant function and $\chi_{\mathrm{Dom}(h)}$ denotes a characteristic function supported on the domain of $h$.
\end{definition}

Hence $\Lc_{s,w,\Ac}$ is well-defined and acts boundedly on $C^1(\Pc)$, and further we have the following spectral properties which are analogous to Proposition \ref{prop:tr:overQ} and stated in  \cite[Theorem D]{kim:lee:lim}.

\begin{proposition} \label{prop:tr:overK}
Set $s=\sigma+it, w=u+iv$. 
\begin{enumerate}
\item The operator $\Lc_{\sigma,u,\Ac}$ on $C^1(\Pc)$ is quasi-compact with a unique simple eigenvalue $\lambda_{\sigma,u,\Ac}$ of maximal modulus associated to eigenfunction $h_{\sigma,u,\Ac}$ with $h_{\sigma,u,\Ac}|_P>0$ for all $P \in \Pc$.
\item The operator $\Lc_{s,w,\Ac}$ depends analytically for $(\sigma,u)\sim (\delta_\Ac, 0)$, accordingly $\lambda_{s,w,\Ac}$ and $h_{s,w,\Ac}$ are well-defined and analytic.
\item The resolvent operator satisfies $\| (\mathcal{I}-\Lc_{s,w,\Ac})^{-1}\|_{(t)} \leq \max\{1, |t|^\xi\}$ uniformly in $w$, with $0<\xi<1/10$ and sufficiently large $|t|$.
\end{enumerate}
\end{proposition}

The proof goes almost the same as in Proposition \ref{prop:tr:overQ} along with uniform expanding and distortion bounds for inverse branches of complex continued fraction maps. We note that the dimension $\delta_\Ac$ takes values between 0 and 2 in this case.

Meanwhile, we remark that there is a cell structure $\Pc=\cup_{i=0}^2 \Pc[i]$ induces a decomposition of the function space $C^1(\Pc)=\bigoplus_{i=0}^2 C^1(\Pc[i])$ and of the operator $\Lc:=\Lc_{s,w,\Ac}$ as a lower-triangular matrix 
\[ \Lc=\begin{bmatrix}
\Lc_{[2]}^{[2]} & 0 & 0 \\
\Lc_{[2]}^{[1]} & \Lc_{[1]}^{[1]} & 0 \\
\Lc_{[2]}^{[0]} & \Lc_{[1]}^{[0]} & \Lc_{[0]}^{[0]} 
\end{bmatrix} \]
where $\Lc^{[i]}_{[j]}: C^1(\Pc[i]) \rightarrow C^1(\Pc[j])$ with $0 \leq i,j \leq 2$ is the component operator. This plays a central role in generalising Proposition \ref{prop:tr:overQ} to the complex case, treating technical issues in a higher dimensional framework.

\section{Complex functions and asymptotic formula} \label{sec:taub}

In this section, we obtain Theorem \ref{main:overQ}.(1). The proof relies on a Tauberian argument, which will allow us to convert analytic properties of the complex function to asymptotic statistics.

We define the two variable series 
\[ L(s,w)= \sum_{n \geq 1} \frac{d_n(w)}{n^s}, \qquad d_n(w):=\sum_{x \in \Sigma_{n,A}} e^{w \cdot 1(x)} \]
for $s,w \in \C$. Note that for all $x=a/n \in \Sigma_{n,A}$, there is a unique $h=h_{a_1} \circ \cdots \cdot h_{a_\ell} \in \Hc_A^\ell$ such that $x=h(0)$ and $n=\mathrm{denom}(x)=|h'(0)|^{1/2}$. Hence we have 
\begin{equation} \label{id:comp}
L(2s,w)=\sum_{h \in \Hc^*} e^{s \log|h'(0)|+w \cdot \sum_{j=1}^\ell 1\circ h_j(0)} = \Lc_{s,w,A}^\sharp \circ \sum_{n \geq 0} \Lc_{s,w,A}^n 1(0),
\end{equation}
where $\Hc^*:=\cup_{n \geq 1} \Hc^n$, and $\Lc_{s,w,A}^\sharp$ denotes the same operator defined by the inverse branches corresponding to $\Hc_A^\sharp:=\{h_a \in \Hc_A: a \geq 2 \} \subset \Hc_A$ for the unique continued fraction expansion for rationals. 

This identity enables us to understand analytic properties of complex series $L(s,w)$ directly follow from Proposition \ref{prop:tr:overQ}: By the implicit function theorem, there is an analytic map $s_0:W \rightarrow \C$ such that for $w \in W$, a complex neighborhood of 0, we have $\lambda_{s_0(w),w,A}=1$. In particular, we have $\lambda_{\delta_A,0,A}=1$ by definition of the pressure function.

\begin{lemma} \label{prop:dirichlet}
For $0<\xi<1/5$, we can find $\a_0$ with the properties that for any $\widehat{\a_0}$ with $0<\widehat{\a_0}<\a_0$ and $w \in W$, we have
\begin{enumerate}
\item $\Re s_0(w)>1-(\a_0-\widehat{\a}_0)$.
\item $L(2s,w)$ has a unique simple pole at $s=s_0(w)$ in the strip $|\Re(s)-\delta_A| \leq \a_0$.
\item $|L(2s,w)| \ll \max\{1, |t|^\xi\}$ in the strip $|\Re(s)-\delta_A| \leq \a_0$, where $t=\Im(s)$.
\end{enumerate}
\end{lemma}

This would be sufficient to apply the following version of truncated Perron's formula; see e.g. Titchmarsh \cite[Lemma 3.19]{titch}.

\begin{theorem}[Perron's Formula with error estimates] \label{perron}
Suppose that $a_n$ is a sequence and $A(x)$ is a non-decreasing function such that $|a_n|=O(A(n))$.
Let $F(s)=\sum_{n \geq 1} \frac{a_n}{n^s}$ for $\Re s:=\sigma>\sigma_a$, the abscissa of absolute convergence of $F(s)$. Then for all $D>\sigma_a$ and $T>0$, one has
\begin{align*}
\sum_{n \leq x} a_n= \frac{1}{2\pi i} \int_{D-iT}^{D+iT} F(s) \frac{x^s}{s}ds &+ O\left(\frac{x^D |F|(D)}{T} \right) +O \left(\frac{A(2x)x\log x}{T} \right) \\ &+ O \left( A(x) \mathrm{min} \left\{ \frac{x}{T|x-M |},1 \right\} \right)
\end{align*}
as $T$ tends to infinity, where \[ |F|(\sigma):=\sum_{n \geq 1} \frac{|a_n|}{n^\sigma} \] for $\sigma > \sigma_a$ and $M$ is the nearest integer to $x$.
\end{theorem}

We are now ready to obtain the asymptotic formula and conclude Theorem \ref{main:overQ}.(1) as follows. This is essentially the same calculation in \cite[Theorem 3.3]{lee:sun}. 

\begin{theorem} \label{thm:mainest}
For a non-vanishing $B(w)$ and $\gamma>0$, we have 
\[ \sum_{n \leq N} d_n(w) =B(w) N^{2 s_0(w)}(1+O(N^{-\gamma})).\]
In particular, we obtain $|\Omega_{N,A}|=\sum_{n \leq N} d_n(0) \sim N^{2\delta_A}$.
\end{theorem}

\begin{proof}
We recall analytic properties of $L(2s,w)$ in Lemma \ref{prop:dirichlet}, which allow us to proceed the contour integration using Cauchy's residue theorem 
\[ \frac{1}{2\pi i} \int_{\mathcal{U}_T(w)} L(2s,w) \frac{N^{2s}}{2s} d(2s)=\frac{E(w)}{s_0(w)} N^{2 s_0(w)}.\]
Here $E(w)$ is the residue of $L(2s,w)$ at the simple pole $s=s_0(w)$ and $\mathcal{U}_T(w)$ is the contour with positive orientation which is a rectangle with the vertices $1+\a_0+iT$, $1-\a_0+iT$, $1-\a_0-iT$, and $1+\a_0-iT$. Applying Perron's formula in Theorem \ref{perron}, we have  
\begin{align*}
\sum_{n \leq N} d_n(w)&= \frac{E(w)}{s_0(w)} N^{2 s_0(w)}+ O\left(\frac{N^{2(1+\a_0)}}{T} \right) +O \left(\frac{A(2N)N\log N}{T} \right) + O(A(N)) \\ &+ O \left( \int_{1-\a_0-iT}^{1-\a_0+iT} |L(2s,w)| \frac{N^{2(1-\a_0)}}{|s|} ds    \right) + O \left( \int_{1-\a_0 \pm iT}^{1+\a_0 \pm iT} |L(2s,w)| \frac{N^{2\Re s}}{T} ds    \right) .
\end{align*} 
The last two error terms are from the contour integral, each of which corresponds to the left vertical line and horizontal lines of the rectangle $\mathcal{U}_T(w)$. We write the expression as 
\[ \sum_{n \leq N} d_n(w)= \frac{E(w)}{s_0(w)} N^{2s_0(w)} ( 1+\mathrm{\rom{1}+ \rom{2}+\rom{3}+\rom{4}+\rom{5}}). \]

Choose $\widehat{\a}_0$ with \[ \frac{32}{79} \a_0 < \widehat{\a}_0 < \a_0 \] and set \[ T= N^{2\a_0+4 \widehat{\a}_0}.\] Notice that $\frac{E(w)}{s_0(w)}$ is bounded in $W$ since we have $s_0(0)=\delta_A$. Then the error terms are bounded as follows:
\begin{itemize}
\item[(\rom{1})] The error term $\mathrm{\rom{1}}$ is equal to $O(N^{2(1-2\widehat{\a}_0-\Re s_0(w))})$. The exponent satisfies \[ 2(1-2\widehat{\a}_0-\Re s_0(w))<2(\a_0-3 \widehat{\a}_0)<0.\]

\item[(\rom{2})] For any $\varepsilon$ with $0<\varepsilon<\frac{\widehat{\a}_0}{2}$, we can choose $W$ in Lemma \ref{prop:dirichlet} small enough to have $k \Re w < \varepsilon/2$ so that $A(N)=O(N^{1+k \Re(w)})=O(N^{1+\varepsilon/2})$ for some $k>0$. Then the exponent of $N$ in the error term $\mathrm{\rom{2}}$ is equal to 
\[ 1-2(\Re s_0(w)+ k \Re w)-(2 \a_0+ 4 \widehat{\a}_0) \leq -2 \a_0+ \frac{5}{2} \widehat{\a}_0 <0. \]

\item[(\rom{3})] Similarly, the error term $\mathrm{\rom{3}}$ is equal to $O(N^{1+k\Re w -2 \Re s_0(w)})$. The exponent satisfies \begin{align*} 1+k \Re w -2 \Re s_0(w) &< -1+2(\a_0+\widehat{\a}_0)+\frac{\varepsilon}{2} \\ &< -1+2\a_0 -\frac{7}{4} \widehat{\a}_0 \leq -\frac{7}{4} \widehat{\a}_0 <0.  \end{align*}

\item[(\rom{4})] For $0<\xi<\frac{1}{5}$, we also have $|L(2s,w)| \ll |t|^{\xi}$. The error term $\mathrm{\rom{4}}$ is $O(N^{2(1-\a_0-\Re s_0(w))} T^\xi)$ and the exponent  is equal to 
\begin{align*} & 2(1-\a_0-\Re s_0(w))+ (2\a_0+4 \widehat{\a}_0)\xi  \\ &< 2(1-\a_0- (1-\a_0+\widehat{\a}_0))+\frac{1}{5}(2\a_0+4 \widehat{\a}_0) \\ &= \frac{2}{5}(\a_0-3 \widehat{\a}_0) <0 .\end{align*}

\item[(\rom{5})] The last term $\mathrm{\rom{5}}$ is $O(T^{\xi-1} \cdot N^{2(1+\a_0-\Re s_0(w))} \log N)$. The exponent satisfies 
\begin{align*} & (2\a_0+4 \widehat{\a}_0)(\xi-1)+2(1+\a_0-\Re s_0(w))+\frac{\varepsilon}{2} \\ &< -\frac{4}{5}(2\a_0+4 \widehat{\a}_0)+ \frac{\widehat{\a}_0}{4}+ 2(2\a_0- \widehat{\a}_0) \\& < \frac{12}{5} \left( \a_0- \frac{99}{48} \widehat{\a}_0 \right)< 0.
\end{align*}
\end{itemize}
By taking \[ \gamma= \mathrm{min} \left(  \frac{7}{4} \widehat{\a}_0, \frac{2}{5}(3\widehat{\a}_0 -\a_0), \frac{12}{5} \left( \frac{99}{48} \widehat{\a}_0 -\a_0 \right)   \right) ,\]
we conclude the theorem.
\end{proof}

\section{Probabilistic smoothing process} \label{sec:smooth}

In this section, we conclude the proof of Theorem \ref{main:overQ}.(2) by adopting the smoothing process and constructing an auxiliary intermediate probability space. 

We define a smoothed probability space $\Sigma_{N,A}(\varepsilon)$ as follows. For $\varepsilon(N)=N^{-\gamma/2}$ and $\gamma>0$ from Theorem \ref{thm:mainest}, consider
\[ \Sigma_{N,A}(\varepsilon) := \bigcup_{n=N-\lfloor N \varepsilon(N) \rfloor}^{N} \Sigma_{n,A}  \]
such that $\Sigma_{N,A} \subset \Sigma_{N,A}(\varepsilon)\subset \Omega_{N,A}$.

We also write $\Psi_w(N)=\sum_{n \leq N} d_n(w)$. Clearly, we have 
\[ \sum_{n=N-\lfloor N \varepsilon(N) \rfloor}^N d_n(w) = \Psi_w(N)-\Psi_w(N-\lfloor N \varepsilon(N) \rfloor) ,\]
and $|\Sigma_{N,A}(\varepsilon)|=\sum_{n=N-\lfloor N \varepsilon(N) \rfloor}^N d_n(0)$. The following smoothing process is similar to the one in \cite[Section 4.2]{bv} and \cite[Section 4]{lee:sun}.

\begin{theorem} \label{thm:mainsm}
With the same setting as in Theorem \ref{thm:mainest}, we have 
\[ \sum_{n=N-\lfloor N \varepsilon(N) \rfloor}^N d_n(w)= 2 \lfloor N \varepsilon(N) \rfloor B(w)\sigma(w)N^{2 s_0(w)-1}(1+O(N^{-\gamma/2})) .\]
Hence we obtain $|\Sigma_{N,A}(\varepsilon)|=\sum_{n=N-\lfloor N \varepsilon(N) \rfloor}^N d_n(0) \sim N^{2\delta_A-\eta}$ for some $0<\eta<1$.
\end{theorem}

\begin{proof}
For simplicity, we write $F_w(N)=B(w)N^{2s_0(w)}$. Then we have
\begin{align*}
\Psi_w(N)-\Psi_w(N-\lfloor N \varepsilon(N) \rfloor) &= [F_w(N)-F_w(N-\lfloor N \varepsilon(N) \rfloor)] + O(F_w(N) N^{-\gamma}) \\ &= \lfloor N \varepsilon(N) \rfloor F_w'(N) + O(F_w(N) N^{-\gamma}) \\ &= \lfloor N \varepsilon(N) \rfloor F_w'(N) \left[ 1+ O \left( \frac{1}{ \lfloor N \varepsilon(N) \rfloor} \cdot \frac{F_w(N) N^{-\gamma}}{F_w'(N)}  \right) \right] .
\end{align*}
Note that $\frac{F_w(N)}{F_w'(N)}=\frac{N}{2 s_0(w)}$ and $s_0(w)$ is bounded and holomorphic in the neighborhood $W$. Since $\varepsilon(N)=N^{-\gamma/2}$, the last error term is equal to $O(N^{-\gamma/2})$ and this concludes the main part of the proof.

Since we have $\lfloor N \varepsilon(N) \rfloor \sim N^\beta$ for some $0<\beta<1$, we obtain the final estimate by taking $\eta:=1-\beta$.
\end{proof}
\bigskip

\bibliographystyle{alpha}
\bibliography{FracStat}

@misc{Shk,
      title={On some results of {K}orobov and {L}archer and {Z}aremba's conjecture}, 
      author={Ilya D. Shkredov},
      year={2026},
      eprint={2603.14116},
      archivePrefix={arXiv},
      primaryClass={math.NT},
      note={https://arxiv.org/abs/2603.14116}, 
}

@article {Shk19,
    AUTHOR = {Moshchevitin, Nikolay and Murphy, Brendan and Shkredov, Ilya},
     TITLE = {Popular products and continued fractions},
   JOURNAL = {Israel J. Math.},
  FJOURNAL = {Israel Journal of Mathematics},
    VOLUME = {238},
      YEAR = {2020},
    NUMBER = {2},
     PAGES = {807--835},
      ISSN = {0021-2172,1565-8511},
   MRCLASS = {11B30 (11A55)},
  MRNUMBER = {4145818},
MRREVIEWER = {Ben\ Joseph\ Green},
       DOI = {10.1007/s11856-020-2039-3},
       URL = {https://doi.org/10.1007/s11856-020-2039-3},
}

@book {titch,
    AUTHOR = {Titchmarsh, E. C.},
     TITLE = {The theory of the {R}iemann zeta-function},
   EDITION = {Second},
      NOTE = {Edited and with a preface by D. R. Heath-Brown},
 PUBLISHER = {The Clarendon Press, Oxford University Press, New York},
      YEAR = {1986},
     PAGES = {x+412},
      ISBN = {0-19-853369-1},
   MRCLASS = {11M06},
  MRNUMBER = {882550},
MRREVIEWER = {Matti\ Jutila},
}

@incollection {jen:pol2,
    AUTHOR = {Jenkinson, Oliver and Pollicott, Mark},
     TITLE = {Rigorous dimension estimates for {C}antor sets arising in
              {Z}aremba theory},
 BOOKTITLE = {Dynamics: topology and numbers},
    SERIES = {Contemp. Math.},
    VOLUME = {744},
     PAGES = {83--107},
 PUBLISHER = {Amer. Math. Soc., Providence},
      YEAR = {2020},
      ISBN = {978-1-4704-5100-4},
   MRCLASS = {11K55 (11A55 11K50 37C30 37F35)},
  MRNUMBER = {4062559},
       DOI = {10.1090/conm/744/14980},
       URL = {https://doi.org/10.1090/conm/744/14980},
}

@article {jen:pol,
    AUTHOR = {Jenkinson, Oliver and Pollicott, Mark},
     TITLE = {Rigorous effective bounds on the {H}ausdorff dimension of
              continued fraction {C}antor sets: a hundred decimal digits for
              the dimension of {$E_2$}},
   JOURNAL = {Adv. Math.},
  FJOURNAL = {Advances in Mathematics},
    VOLUME = {325},
      YEAR = {2018},
     PAGES = {87--115},
      ISSN = {0001-8708,1090-2082},
   MRCLASS = {11K55 (30B70 30H10 37C30 37C45 40A15 47B06)},
  MRNUMBER = {3742587},
MRREVIEWER = {Dieter\ H.\ Mayer},
       DOI = {10.1016/j.aim.2017.11.028},
       URL = {https://doi.org/10.1016/j.aim.2017.11.028},
}

@article {ref1,
    AUTHOR = {Gonz\'alez Robert, Gerardo and Hussain, Mumtaz and Shulga,
              Nikita},
     TITLE = {Complex numbers with a prescribed order of approximation and
              {Z}aremba's conjecture},
   JOURNAL = {J. Number Theory},
  FJOURNAL = {Journal of Number Theory},
    VOLUME = {274},
      YEAR = {2025},
     PAGES = {1--25},
      ISSN = {0022-314X,1096-1658},
   MRCLASS = {11J70 (28A78 30B70)},
  MRNUMBER = {4871630},
MRREVIEWER = {A.\ Peth\H o},
       DOI = {10.1016/j.jnt.2024.12.010},
       URL = {https://doi.org/10.1016/j.jnt.2024.12.010},
}

@article {ref2,
    AUTHOR = {Sarkar, Pratyush},
     TITLE = {Congruence counting in {S}chottky and continued fractions
              semigroups of {SO}({$n$}, 1)},
   JOURNAL = {J. Anal. Math.},
  FJOURNAL = {Journal d'Analyse Math\'{e}matique},
    VOLUME = {157},
      YEAR = {2025},
    NUMBER = {2},
     PAGES = {617--672},
      ISSN = {0021-7670,1565-8538},
   MRCLASS = {Prelim},
  MRNUMBER = {4999825},
       DOI = {10.1007/s11854-025-0425-9},
       URL = {https://doi.org/10.1007/s11854-025-0425-9},
}

@article {hurwitz,
    AUTHOR = {Hurwitz, Adolf},
     TITLE = {\"{U}ber die {E}ntwicklung complexer {G}r\"{o}ssen in {K}ettenbr\"{u}che},
   JOURNAL = {Acta Math.},
  FJOURNAL = {Acta Mathematica},
    VOLUME = {11},
      YEAR = {1887},
    NUMBER = {1-4},
     PAGES = {187--200},
      ISSN = {0001-5962},
   MRCLASS = {DML},
  MRNUMBER = {1554754},
       DOI = {10.1007/BF02418048},
       URL = {https://doi.org/10.1007/BF02418048},
}

@article {bour:kon,
    AUTHOR = {Bourgain, Jean and Kontorovich, Alex},
     TITLE = {On {Z}aremba's conjecture},
   JOURNAL = {Ann. of Math. (2)},
  FJOURNAL = {Annals of Mathematics. Second Series},
    VOLUME = {180},
      YEAR = {2014},
    NUMBER = {1},
     PAGES = {137--196},
      ISSN = {0003-486X},
   MRCLASS = {11J70 (11A55)},
  MRNUMBER = {3194813},
MRREVIEWER = {Oto Strauch},
       DOI = {10.4007/annals.2014.180.1.3},
       URL = {https://doi.org/10.4007/annals.2014.180.1.3},
}

@article {dolgopyat,
    AUTHOR = {Dolgopyat, Dmitry},
     TITLE = {On decay of correlations in {A}nosov flows},
   JOURNAL = {Ann. of Math. (2)},
  FJOURNAL = {Annals of Mathematics. Second Series},
    VOLUME = {147},
      YEAR = {1998},
    NUMBER = {2},
     PAGES = {357--390},
      ISSN = {0003-486X},
   MRCLASS = {58F11 (58F15)},
  MRNUMBER = {1626749},
MRREVIEWER = {Luis M. Barreira},
       DOI = {10.2307/121012},
       URL = {https://doi.org/10.2307/121012},
}

@article {hensley,
    AUTHOR = {Hensley, Douglas},
     TITLE = {A polynomial time algorithm for the {H}ausdorff dimension of
              continued fraction {C}antor sets},
   JOURNAL = {J. Number Theory},
  FJOURNAL = {Journal of Number Theory},
    VOLUME = {58},
      YEAR = {1996},
    NUMBER = {1},
     PAGES = {9--45},
      ISSN = {0022-314X},
   MRCLASS = {11K55 (28A78)},
  MRNUMBER = {1387719},
MRREVIEWER = {Oto Strauch},
       DOI = {10.1006/jnth.1996.0058},
       URL = {https://doi.org/10.1006/jnth.1996.0058},
}

@article {nakada,
     AUTHOR = {Ei, Hiromi and Nakada, Hitoshi and Natsui, Rie},
title = {On the ergodic theory of maps associated with the nearest integer complex continued fractions over imaginary quadratic fields},
journal = {Discrete and Continuous Dynamical Systems},
volume = {43},
number = {11},
pages = {3883-3924},
year = {2023},
}

@article {par:pol,
    AUTHOR = {Parry, William and Pollicott, Mark},
     TITLE = {Zeta functions and the periodic orbit structure of hyperbolic
              dynamics},
   JOURNAL = {Ast\'erisque},
  FJOURNAL = {Ast\'erisque},
    VOLUME = {187-188},
      YEAR = {1990},
     PAGES = {pp.268},
      ISSN = {0303-1179,2492-5926},
   MRCLASS = {58F20 (58F11 58F15)},
  MRNUMBER = {1085356},
MRREVIEWER = {Nicola\u i\ T.\ A.\ Haydn},	
}

@article {lee:sun,
    AUTHOR = {Lee, Jungwon and Sun, Hae-Sang},
     TITLE = {Another note on ``{E}uclidean algorithms are {G}aussian'' by
              {V}. {B}aladi and {B}. {V}all\'{e}e},
   JOURNAL = {Acta Arith.},
  FJOURNAL = {Acta Arithmetica},
    VOLUME = {188},
      YEAR = {2019},
    NUMBER = {3},
     PAGES = {241--251},
      ISSN = {0065-1036},
   MRCLASS = {37A45 (11Y16 37C30)},
  MRNUMBER = {3928686},
       DOI = {10.4064/aa170418-6-3},
       URL = {https://doi.org/10.4064/aa170418-6-3},
}

@article {bv,
    AUTHOR = {Baladi, Viviane and Vall\'{e}e, Brigitte},
     TITLE = {Euclidean algorithms are {G}aussian},
   JOURNAL = {J. Number Theory},
  FJOURNAL = {Journal of Number Theory},
    VOLUME = {110},
      YEAR = {2005},
    NUMBER = {2},
     PAGES = {331--386},
      ISSN = {0022-314X},
   MRCLASS = {11Y16 (37A45 37C30 60F05 68W40)},
  MRNUMBER = {2122613},
MRREVIEWER = {Jeffrey O. Shallit},
       DOI = {10.1016/j.jnt.2004.08.008},
       URL = {https://doi.org/10.1016/j.jnt.2004.08.008},
}

@book {baladi,
    AUTHOR = {Baladi, Viviane},
     TITLE = {Positive transfer operators and decay of correlations},
    SERIES = {Advanced Series in Nonlinear Dynamics},
    VOLUME = {16},
 PUBLISHER = {World Scientific Publishing Co., Inc., River Edge, NJ},
      YEAR = {2000},
     PAGES = {x+314},
      ISBN = {981-02-3328-0},
   MRCLASS = {37C30 (37-02 37A25 37D20 47B38)},
  MRNUMBER = {1793194},
MRREVIEWER = {J\'{e}r\^{o}me Buzzi},
       DOI = {10.1142/9789812813633},
       URL = {https://doi.org/10.1142/9789812813633},
}

@article {kim:lee:lim,
    AUTHOR = {Kim, Dohyeong and Lee, Jungwon and Lim, Seonhee},
     TITLE = {Euclidean algorithms are {G}aussian over imaginary quadratic
              fields},
   JOURNAL = {J. Lond. Math. Soc. (2)},
  FJOURNAL = {Journal of the London Mathematical Society. Second Series},
    VOLUME = {112},
      YEAR = {2025},
    NUMBER = {6},
     PAGES = {Paper No. e70333, 45},
      ISSN = {0024-6107,1469-7750},
   MRCLASS = {11J70 (11R04 37C30 60F05)},
  MRNUMBER = {5005061},
       DOI = {10.1112/jlms.70333},
       URL = {https://doi.org/10.1112/jlms.70333},
}

\end{document}